\def\obs#1{{\bf (*** #1 ***)} }

\def\obs#1{}     
\documentclass[twoside,letterpaper,draft,11pt]{amsart}
\usepackage{amsmath,amsthm,latexsym,xspace,amscd,amssymb,mathtools,color,enumerate}
\usepackage{pb-diagram}
\usepackage{tikz-cd}

\DeclareUnicodeCharacter{00B4}{}

\usepackage[utf8]{inputenc}

\usepackage[mathscr]{eucal}
\usepackage{amssymb}
\usepackage[all]{xy}
\makeatletter\renewcommand\theenumi{\@roman\c@enumi}\makeatother

\newcommand{\m}{{}^{-1}}

\newtheorem{thm}{Theorem}[section]
\newtheorem{cor}[thm]{Corollary}
\newtheorem{lem}[thm]{Lemma}
\newtheorem{prop}[thm]{Proposition}
\newtheorem{defn}[thm]{Definition}
\newtheorem{rem}[thm]{Remark}
\newtheorem{exe}[thm]{Example}
\newcommand{\Ga}{\Gamma}
\newcommand{\C}{\mathcal{C}}
\newcommand{\D}{\mathcal{D}}
\newcommand{\K}{\mathcal{K}}
\newcommand{\F}{{\bf F}}
\newcommand{\Gg}{{\bf G}}
\newcommand{\Hh}{\mathcal{H}}

\newcommand{\om}{\omega}

\newcommand{\gra}{\textbf{Gr}(\alpha)}
\newcommand{\pacg}{\textbf{Pact}(\G)}
\newcommand{\gacg}{\textbf{Glob}(\G)}
\newcommand{\spacg}{\textbf{StrictPact}(\G)}
\newcommand{\sgacg}{\textbf{StrictGact}(\G)}
\newcommand{\stig}{\textbf{Starinj}(\G)}
\newcommand{\cov}{\textbf{Cov}(\G)}

\newcommand{\id}{{\rm id}}

\newcommand{\G}{\mathcal{G}}

\def\ndv{\ {\mid \kern -0.7 em {\scriptstyle \not}} \ \ }

\def\nd{\ {\mid \kern -0.4 em {\scriptstyle \not}} \ \ }

\begin{document}

\thispagestyle{empty}

\title[On partial Groupoid Actions]{On direct product, semidirect product of Groupoids and partial actions}

\author[V. Mar\'in ]{V\'ictor Mar\'in}
\address{Departamento de Matem\'{a}ticas y Estad\'{i}stica, Universidad del Tolima, Santa Helena\\
	Ibagu\'{e}, Colombia} \email{vemarinc@ut.edu.co}

\author[H. Pinedo]{H\'ector Pinedo}
\address{Escuela de Matematicas, Universidad Industrial de Santander, Cra. 27 Calle 9 ´ UIS
	Edificio 45, Bucaramanga, Colombia}
\email{hpinedot@uis.edu.co}


\date{\today}
\begin{abstract}
We present some constructions of groupoids as: direct product, semidirect product, and we give necessary and sufficient conditions  for  a groupoid to be embedded into a direct product of groupoids. Also, we establish necessary and sufficient conditions  to determine when a semidirect product is direct. Moreover, we establish an  equivalence between the category of  strict partial groupoid actions and the category of star injective functors. Finally, we give a relation of categorical type between the  actions groupoids $(\G,X)$  and $(\G,Y),$ being $Y$ a universal globalization of $X.$ 
 \end{abstract}

\maketitle

\setcounter{tocdepth}{1}

\section{Introduction}
A groupoid is a small category in which every morphism is invertible. An equivalent  notion of groupoid  from an axiomatic approach, such as that of  a group, is presented in \cite[p. 89]{L}. Using this point of view,  Paques and Tamusiunas gave necessary and sufficient conditions for a subgroupoid  to be a normal  (see Definition \ref{nor}) and  construct the quotient groupoid \cite{PT} which played an important role in the study of Galois extensions for groupoid actions.  
Groupoids have also appear in the context of partial actions of groups. For instance  in \cite{KL} the authors presented an  alternative way to study globalizations through groupoids,  in \cite{A2}  it is shown that  partial action gives rise to a groupoid provided with a Haar system, whose $C^*$-algebra agrees with the crossed product by the partial action and in \cite{A1} several groupoids related to partial actions of discrete groups are studied. 
Later  in \cite{G} Gilbert introduced partial actions of ordered groupoids in terms of ordered premorphism of groupoids  and generalized some of the result obtained in \cite{KL}   to the frame of partial groupoid actions. 
In \cite{NY}, Nysted gave the definition of partial action of groupoid on a set and shows that all partial action of groupoids posses a universal globalization. Moreover, partial actions of groupoids on other classes of objects rather than sets have been studied, for instance in \cite{BP} the authors introduced the notion of a partial action of an arbitrary groupoid on a ring, while in \cite{MP} this partial actions were considered on $R$-semicategories, and a conexion with inverse categories was presented in \cite{NOP}.

The first  goal of this paper is to present some new constructions  of groupoids and study some structural  properties of them. For this, after the introduction  in Section \ref{prel}  we introduce the necessary background on  groupoids. In Section \ref{product} we define direct product of an arbitrary family  of groupoids and give in Theorem \ref{caracterizacionproductodirecto} necessary and sufficient conditions for a groupoid to be embedded in a direct product of groupoids. At this point it is important to remark that in \cite{PL} the author deals with internal direct product of groupoids, but in his work groupoids are considered  as binary system so his approach is different from ours.
 We also construct a  groupoid  induced from an action of a group on a groupouid and give in Theorem \ref{semipro} necessary and suficient conditions for it to be a direct product of groupoids.

Our other goal is related to partial action of groupoids, for this in Section \ref{paction} we deal with partial actions, in particular, inspired by the results of \cite[Section 3.1]{KL}  we determine in Theorem \ref{equiv} an equivalence between the categories of strict partial actions of a groupoid $\G$ 
and the category of star injective functors on $\G,$ and also study globalization of partial groupoid actions using this point of view.  



\medskip

\section{Preliminaries}\label{prel}

Recall that a {\it groupoid} is a small category in which every morphism is an isomorphism. The set of the objects of a groupoid $\G$ will be denoted by $\G_0$.
If $g:e\to f$ is a morphism of $\G$ then $d(g)=e$ and $r(g)=f$ are called the {\it domain} and the {\it range} of $g$ respectively. We will identify any object $e$ of $\G$ with its identity morphism, that is, $e=\id_e$.  
The {\it isotropy group} associated to an object $e$ of $\G$ is the group $\G_e=\{g\in \G \mid\,d(g)=r(g)=e\}$. The set ${\rm Iso}(\G)=\bigcup_{e\in \G_0}\G_e$  is called the {\it isotropy subgroupoid} of $\mathcal{G}.$ \vspace{.05cm}

The composition of morphisms of a groupoid $\G$ will be denoted via concatenation. Hence, for $g,h\in \G$, there exists $gh$, denoted $\exists gh$, if and only if $r(h)=d(g)$.
Notice that, if $g\in \G$ then its inverse  $g^{-1}$  is unique,  $d(g)=g^{-1}g$ and $r(g)=gg\m$.

The following is well known (see for instance \cite[Proposition 2.7]{AMP}).

\begin{prop}\label{t1}
Let $\mathcal{G}$ be a groupoid, $g,h \in \G$ and $n\in \mathbb{N},$  the following assertions hold.
\begin{enumerate}
\item[(i)] If $\exists gh$, then $d(gh)=d(h)$ and $r(gh)=r(g)$.
\item[(iii)] $\exists gh$ if, and only if, $\exists h^{-1}g^{-1}$ and, in this case, $(gh)^{-1}=h^{-1}g^{-1}$.
\item[(iii)] For $g_1, g_2, \dots, g_n\in \G$, $\exists g_1g_2\cdots g_n$ if and only if $r(g_{i+1})=d(g_i),$  for $1\leq i\leq n-1.$
\end{enumerate}
\end{prop}

\vspace{.2cm}

Recall the notions  of subgroupoid, wide and normal subgroupoid. 

\begin{defn}\cite[p. 107]{PT}.
Let $\mathcal{G}$ be a groupoid and $\mathcal{H}$ a nonempty subset of $\mathcal{G}$.  $\mathcal{H}$ is said to be a subgroupoid of $\mathcal{G}$ if for all $g,h \in \mathcal{H}$  $g^{-1}\in \mathcal{H}$ and $gh\in \mathcal{H}$ provided that $\exists gh.$  In this case we denote $\mathcal{H}\leq \mathcal{G}$. If in addition $\mathcal{H}_0=\mathcal{G}_0$ (or equivalently $\mathcal{G}_0\subseteq \mathcal{H}$) then $\mathcal{H}$ is called a wide subgroupoid of $\mathcal{G}$.

\end{defn}

If $\mathcal{H}$ is a wide subgroupoid of $\mathcal{G}$ and $g\in \mathcal{G}$ then $$g^{-1}\mathcal{H}g=\{g^{-1}hg\mid h\in \mathcal{H}\textnormal{ and } r(h)=d(h)=r(g)\}$$ is a subgroupoid of $\mathcal{G}$. 

\begin{defn}\label{nor}
Let $\mathcal{G}$ be a groupoid. A subgroupoid $\mathcal{H}$ of $\mathcal{G}$
is said to be normal, denoted by $\mathcal{H}\lhd \mathcal{G}$, if $\mathcal{H}$ is wide and $g^{-1}\mathcal{H}g\subseteq \mathcal{H},$ for all $g\in \mathcal{G}$.
\end{defn}

Note that $\mathcal{G}$ is a normal  subgroupoid of $\mathcal{G}$. Now, $\mathcal{G}_0$ is a wide subgroupoid of $\mathcal{G}$ and if $g\in \mathcal{G}$ then $g^{-1}\mathcal{G}_0g=\{d(g)\}\subseteq \mathcal{G}_0$. That is, $\mathcal{G}_0$ is also a normal subgroupoid of $\mathcal{G}$. 
\begin{rem}
{\rm Normality was defined in \cite{Br} as follows: A subgroupoid $\mathcal{H}$ is said to be normal if $\mathcal{H}_0=\mathcal{G}_0$ and $g^{-1}\mathcal{H}_{r(g)}g=\mathcal{H}_{d(g)},$ for all $g\in \mathcal{G}$. The equivalence between the two definitions of normality appears  in \cite[Lemma 3.1]{PT}.}
\end{rem}

\vspace{.3cm}

If $\mathcal{G}$ and $\mathcal{G'}$ are groupoids and  $\F: \mathcal{G}\to \mathcal{G'}$ a functor,   it is not difficult to show that if $\K$ is a wide subgroupoid of $\G',$ then $\F^{-1}(\K)=\{x\in\G\mid \F(x)\in \K\}$ is a wide subgroupoid of $\G.$ 
We say that $\F$ is a monomorphism if it is injective, 
 in this case   $\G$ is {\it embedded } in $\G'.$

\section{Direct and semidirect product of groupoids}\label{product}
We start  this section by giving give the definition of the direct product of an arbitrary family of groupoids and obtain a criteria to decide when a  groupoid is embedded to a direct product of groupoids. Later semidirect products are also considered.
\subsection{The direct product of a family of groupoids}

Let $\{\mathcal{G}_i \mid i\in I\}$ be a family of groupoids and  $\prod_{i\in I}\mathcal{G}_i$ the direct product of the family of sets $\{\mathcal{G}_i \mid i \in I\}$. We define a partially binary operation on $\{\mathcal{G}_i \mid i \in I\}$ as follows. Given $(x_i)_{i\in I}, (y_i)_{i\in I} \in \prod_{i\in I}\mathcal{G}_i$,
\begin{equation}\label{directed}
\exists (x_i)_{i\in I}(y_i)_{i\in I} \Longleftrightarrow \exists x_iy_i\, \, \forall i\in I \land (x_i)_{i\in I}(y_i)_{i\in I} =(x_iy_i)_{i\in I}.
\end{equation}
It is clear that with the  product  \eqref{directed} the set $\prod_{i\in I}\mathcal{G}_i$ is a groupoid. Let  $n$ be a natural number , then  if $I$ is a set with $n$ elements and $\G_i=\G$ for all $i\in I,$ the set $\prod_{i\in I}\mathcal{G}_i$ is denoted by $\G^n.$ 

Now consider  a family $\{X_i\}_{1\leq i\leq n}$  consisting of non-empty subsets of $\G.$ We set
$$(X_1\times X_2\times \cdots \times X_n)^{(n)}=\{(x_1, x_2, \dots, x_n)\in X_1\times X_2\times \cdots \times  X_n\mid \exists x_1\cdots x_n \}$$ and
$$X_1\cdots X_n=\{ x_1\cdots x_n\mid (x_1, x_2, \dots, x_n)\in (X_1\times X_2\times \cdots \times X_n)^{(n)}\}.$$ 

In general it is not true that $(X_1\times X_2\times \cdots \times X_n)^{(n)}$ is a subgroupoid of $\G^n,$ even though each $X_i$ is a subgroupoid of $\G,$ for all $1\leq i \leq n.$ In the next result we provide necessary and sufficient conditions for  $(X_1\times X_2\times \cdots \times X_n)^{(n)}$ to be a groupoid.

\begin{prop}\label{restricted} Let $n$ be a natural number and $\Hh_i$ be a subgroupoid of $\G,$ for   $i\in \{1, \dots, n\}.$ Then $(\mathcal{H}_1\times \cdots \times \mathcal{H}_n ) ^{(n)}$ is a subgroupoid of $\G^n$ if and only if $r(h_i)=d(h_{i+1})$ for  $1\leq i\leq n-1$ with $(h_1, \dots, h_n)\in (\mathcal{H}_1\times \cdots \times \mathcal{H}_n ) ^{(n)}$.
\end{prop}
\begin{proof} $(\Rightarrow)$  Let $(h_1, \dots, h_n)\in (\mathcal{H}_1\times \cdots \times \mathcal{H}_n) ^{(n)}$. We need to show that $d(h_{i+1})=r(h_i),$ for all $1\leq i \leq n-1.$ By assumption  $d(h_i)=r(h_{i+1}),$ for $1\leq i\leq n-1.$ Since  $(\mathcal{H}_1\times \cdots \times \mathcal{H}_n ) ^{(n)}$ is a groupoid we have $(h^{-1}_1, \dots, h^{-1}_n)\in (\mathcal{H}_1\times \cdots \times \mathcal{H}_n ) ^{(n)}$, which gives $r(h_i)=d(h_{i+1})$ for $1\leq i\leq n-1.$ 

$(\Leftarrow)$   Let $(h_1, \dots, h_n)\in (\mathcal{H}_1\times \cdots \times \mathcal{H}_n ) ^{(n)}$ and  $1\leq i\leq n-1,$ then $r(h^{-1}_{i+1})=d(h^{-1}_{i+1})=r(h_i)=d(h^{-1}_{i+1})$ and thus $(h^{-1}_1, \dots, h^{-1}_n)\in (\mathcal{H}_1\times \cdots \times \mathcal{H}_n ) ^{(n)}$. Now if  $(h'_1, \dots, h'_n)\in (\mathcal{H}_1\times \cdots \times \mathcal{H}_n ) ^{(n)}$  and $\exists(h_1, \dots, h_n)(h'_1, \dots, h'_n)$ then 
$(h_1h'_1, \dots, h_nh'_n)\in \mathcal{H}_1\times \cdots \times \mathcal{H}_n $. Now for $1\leq i\leq n-1$ $$d(h'_i)=r(h'_{i+1})=d(h_i)=r(h_{i+1})$$ and thus $\exists h'_ih_{i+1},$ which implies $(h_ih'_1, \dots, h_nh'_n)\in (\mathcal{H}_1\times \cdots \times \mathcal{H}_n)^{(n)},$ and we conclude that  $(\mathcal{H}_1\times \cdots \times \mathcal{H}_n ) ^{(n)}$ is a subgroupoid of $\G^n.$ 
\end{proof}
\begin{cor}  Let $n$ be a natural number and $\Hh_i$ be wide  subgroupoids of $\G,$ for   $i\in \{1, \dots, n\}.$ Then $(\mathcal{H}_1\times \cdots \times \mathcal{H}_n ) ^{(n)}$ is a subgroupoid of $\G^n$ if and only if $\mathcal{H}_i={\rm Iso\,\mathcal{H}}_i$ for  $1\leq i\leq n.$ 
\end{cor}
\begin{proof} $(\Rightarrow)$ Take $h_i\in \mathcal{H}_i$ then $(r(h_i),\dots, r(h_i), h_i, d(x_i), \dots, d(x_i) )\in (\mathcal{H}_1\times \cdots \times \mathcal{H}_n ) ^{(n)},$ by  Proposition \ref{restricted} we have that $r(r(h_i))=d(h_i)=d(d(h_i)),$ that is $r(h_i)=d(h_i)$ and $h_i\in {\rm Iso\,\mathcal{H}}_i,$ that is $\mathcal{H}_i={\rm Iso\,\mathcal{H}}_i.$

$(\Leftarrow)$ Let $(h_1, \dots, h_n)\in (\mathcal{H}_1\times \cdots \times \mathcal{H}_n ) ^{(n)}$ then $r(h_i)=d(h_i)=r(h_{i+1})=d(h_{i+1}),$  again by Proposition \ref{restricted}  we have that $(\mathcal{H}_1\times \cdots \times \mathcal{H}_n ) ^{(n)}$ is a subgroupoid of $\G^n$.
\end{proof}

We have the next.
\begin{prop}\label{caracterizacionproductodirecto2}
Let $\mathcal{G}$ be a groupoid and $\mathcal{H}_1,\dots , \mathcal{H}_n$ be wide subgroupoids of $\G$.  Consider the following assertions.
\begin{enumerate}
\item[(i)] $\mathcal{G}=\mathcal{H}_1\cdots \mathcal{H}_n$.
\item[(ii)] $\mathcal{H}_i \lhd \mathcal{G}, \, \, \forall i=1,\dots ,n$.
\item[(iii)] $\mathcal{H}_i\cap (\mathcal{H}_1\cdots \mathcal{H}_{i-1}\mathcal{H}_{i+1}\cdots \mathcal{H}_n)=\mathcal{G}_0, \,\, \forall i=1,\dots , n$.
\item[(iv)] For each $g\in \mathcal{G}$, there exist unique elements $x_1\in \mathcal{H}_1, \dots , x_n \in \mathcal{H}_n$ such that $(x_1, \dots, x_n)\in (\Hh_1\times \Hh_2\times \cdots \times \Hh_n)^{(n)} $ and $g=x_1\cdots x_n$.
\item[(v)] For each $i \ne j$, we have that $xy=yx, \, \, \forall x\in \mathcal{H}_i$ and $\forall y\in \mathcal{H}_j$ such that $$r(y)=d(x)=r(x)=d(y).$$
\end{enumerate}
Then (i),(ii) and (iii) hold if and only if (iv) and (v) hold.
\end{prop}

%

\begin{proof}
Assume that the conditions $(i),\, (ii),\, (iii)$ hold. We start by checking (v), for this  take $x\in \mathcal{H}_i$ and $y\in \mathcal{H}_j$, with $i\ne j$, and $r(y)=d(x)=r(x)=d(y)$, then $\exists y^{-1}x^{-1}yx$. We have $y^{-1}x^{-1}yx=(y^{-1}x^{-1}y)x\in \mathcal{H}_i$, because $\mathcal{H}_i \lhd \mathcal{G}$, and $y^{-1}x^{-1}yx=y^{-1}(x^{-1}yx)\in \mathcal{H}_j$, because $\mathcal{H}_j \lhd \mathcal{G}$. Thus, $y^{-1}x^{-1}yx\in \mathcal{H}_i\cap \mathcal{H}_j=\mathcal{G}_0$. that is, 
\begin{equation}\label{despejar}y^{-1}x^{-1}yx=d(g),\end{equation} for some $g\in \mathcal{G}$. Then, $d(x)=d(y^{-1}x^{-1}yx)=d(g)$ and thus $r(y)=d(y)=r(x)=d(x)=d(g)$.  From this  and \eqref{despejar} we get $xy=yx$
which is (v).

Now we prove  (iv). Take $g\in \mathcal{G}$  then by (i) there are $x_i\in \mathcal{H}_i, \, 1\leq i\leq n$ such that  $g=x_1\cdots x_n$. If   $g=y_1\cdots y_n,$ with $y_i\in \mathcal{H}_i$.  Let $1\leq i\leq n-1$ then $d(x_i)=r(x_{i+1})$ and $d(y_i)=r(y_{i+1}),$ also  $\exists y^{-1}_1(y_1\cdots y_n)$ and $\exists x_1\cdots x_nx^{-1}_n\cdots x^{-1}_2$, thus
$$y_1^{-1}x_1=y_2y_3\cdots y_{n-1}y_nx_n^{-1}x_{n-1}^{-1}\cdots x_2^{-1}.$$
By  (v) we have
$$y_1^{-1}x_1=y_2x_2^{-1}y_3x_3^{-1}\cdots y_nx_n^{-1},$$
where $y_i^{-1}x_i\in \mathcal{H}_i$ with $i=1,\dots, n$. Thus, $y_1^{-1}x_1\in \mathcal{H}_1\cap (\mathcal{H}_2\cdots \mathcal{H}_n)=\mathcal{G}_0$ and there exists $g\in \mathcal{G}$ such that $y_1^{-1}x_1=d(g)$. Hence, $x_1=y_1$.\\
Now, as $x_1\cdots x_n=y_1\cdots y_n$ and $x_1=y_1$ by  the cancellation law for groupoids, we obtain $x_2\cdots x_n=y_2\cdots y_n$ and  continuing  this process we have $x_i=y_i$ for $i=2,\dots , n$.

For the other implication, suppose that  conditions (iv) and (v) are satisfied. First of all note  that  (i) follows from (iv). To show (ii) take $i\in\{1,\dots n\}$ since $\mathcal{H}_i$ is wide subgroupoid, then $\mathcal{G}_0=(\mathcal{H}_i)_0$ and thus $g^{-1}\mathcal{H}_ig\ne \emptyset$ for all $i=1,\dots n$ and $g\in \G$. We will see that given $g\in \mathcal{G}$ and $y\in \mathcal{H}_i$ such that  $\exists g^{-1}yg$ then $g^{-1}yg\in \mathcal{H}_i$. Indeed, by (iv), there are $x_j\in \mathcal{H}_j, 1\leq j\leq n$ such that  $\exists x_1x_2\cdots x_n$ and $g=x_1x_2\cdots x_n.$  Then 
$$g^{-1}yg=x_n^{-1}\cdots x_i^{-1}x_{i-1}^{-1}\cdots x_1^{-1}yx_1\cdots x_{i-1}x_ix_{i+1}\cdots x_n.$$
Furthermore, for $j=1,\dots, i-1$, $yx_j=x_jy$ and $d(x_j)=r(y)$ thanks to  (v). Thus,
$$x_j^{-1}yx_j=x_j^{-1}x_jy=d(x_j)y=r(y)y=y,\quad \text{for all}\quad j=1,\dots , i-1.$$
Hence, $g^{-1}yg=x_n^{-1}\cdots x_i^{-1}yx_ix_{i+1}\cdots x_n$. Now, as for all $j=i+1,\dots, n-1$, $x^{-1}_jx^{-1}_{j+1}=x^{-1}_{j+1}x^{-1}_j$ and $x_jx_{j+1}=x_{j+1}x_j$, we get
\begin{equation*}
g^{-1}yg=x_i^{-1}x_{i+1}^{-1}\cdots x_n^{-1}yx_n\cdots x_{i+1}x_i.
\end{equation*}
Now, for $j=i,\dots, n$, $yx_j=x_jy$ and $d(x_j)=r(y)$ by condition (v). It implies that
$$x_j^{-1}yx_j=x_j^{-1}x_jy=d(x_j)y=r(y)y=y,\quad \text{for all}\quad j=i+1,\dots , n.$$
Thus, $g^{-1}yg=x_i^{-1}yx_i\in \mathcal{H}_i$ which is (ii).

To fininsh the proof, we show  (iii).  As $\mathcal{G}_0=(\mathcal{H}_i)_0$ for all $i=1,\dots , n$, we have that $\mathcal{G}_0\subseteq \mathcal{H}_i\cap (\mathcal{H}_1\cdots \mathcal{H}_{i-1}\mathcal{H}_{i+1}\cdots \mathcal{H}_n)$.  Now take $g\in \mathcal{H}_i\cap (\mathcal{H}_1\cdots \mathcal{H}_{i-1}\mathcal{H}_{i+1}\cdots \mathcal{H}_n)$, then $g\in \mathcal{H}_i$, and 
$$g=x_1\cdots x_{i-1} g x_{i+1}\cdots x_n$$
 with $x_j=r(g)$ for $j=1,\dots, i-1$ and $x_k=d(g)$ for $k=i+1,\dots, n.$  On the other hand, as $g\in \mathcal{H}_1\cdots \mathcal{H}_{i-1}\mathcal{H}_{i+1}\cdots \mathcal{H}_n$, then $$g=h_1\cdots h_{i-1}h_{i+1}\cdots h_n= x_1\cdots x_{i-1}h_ih_{i+1}\cdots x_n$$
with $h_i=d(x_{i-1}),$ then $x_j, h_j\in \mathcal{H}_j $ for all  $j=1, \dots, n$ and by (iv) we get $g\in \mathcal{G}_0,$ as desired.
\end{proof}

\begin{thm}\label{caracterizacionproductodirecto}
Let $\mathcal{G}$ be a  groupoid,  then $\mathcal{G}$ is embedded in a direct product of groupoids if and only if there are   $\mathcal{H}_1,\ldots , \mathcal{H}_n$ subgroupoids of $\G$ such that     $\mathcal{H}_1,\ldots , \mathcal{H}_n$ satisfy conditions (i), (ii) and (iii) of Proposition \ref{caracterizacionproductodirecto2}. 
\end{thm}

\begin{proof}
($\Rightarrow $) Suppose that there exist a groupoid monomorphism $\F: \mathcal{G}\to \mathcal{G}_1\times \cdots \times \mathcal{G}_n.$   Since 
$$\mathcal{K}_i=(\mathcal{G}_1)_0\times \cdots \times \mathcal{G}_i\times \cdots \times (\mathcal{G}_n)_0$$
is a normal subgroupoid of $\mathcal{G}_1\times \cdots \times \mathcal{G}_n$, then  the family of groupoids $\mathcal{H}_i=\F^{-1}(\mathcal{K}_i)$ for $i=1,\dots , n$ satisfies  (i), (ii) and (iii) of  Proposition \ref{caracterizacionproductodirecto2}. 

($\Rightarrow $) Suppose that there exist subgroupoids $\mathcal{H}_i$ of $\mathcal{G}, 1\leq i\leq n$ that satisfy conditions (i), (ii) and (iii). Consider 
the map
\begin{equation}\label{embed}\F: \mathcal{G}=\mathcal{H}_1\cdots \mathcal{H}_n\ni (h_1\cdots h_n)\mapsto (h_1,\dots, h_n)\in \G^n\end{equation}
  Note that  conditio (iv) makes $\F$ is well defined and injective.  We prove that $\F$ is a functor.  Suppose that $\exists x_1x_2\cdots x_n, \exists y_1y_2\cdots y_n$ and  $g=x_1x_2\cdots x_n$ and $g'=y_1y_2\cdots y_n$ in $\mathcal{G}$. Then, if $\exists gg'$ we have 
$$gg'=x_1x_2\cdots x_ny_1y_2\cdots y_n=x_1y_1x_2y_2\cdots x_ny_n.$$ 
Using  condition (v) we get $\F(gg')=(x_1y_1,x_2y_2,\dots , x_ny_n)=\F(g)\F(g').$
\end{proof}
\begin{rem} {\rm The image of $\F$ defined by \eqref{embed} is  $(\mathcal{H}_1\times \cdots \times \mathcal{H}_n ) ^{(n)},$ then under conditions of  Theorem \ref{caracterizacionproductodirecto} \, $\G$ is isomorphic to a subgroupoid of a direct product, provided that the groupoids $\Hh_i$ satisfy the assumptions in Proposition \ref{restricted}. }
\end{rem}
 \begin{exe} {\rm Consider the groupoid $\G=\{a,u,v,u^{-1}, v^{-1}, x,y\},$ where $d(a)=r(a)=x, d(v)=r(u)=y$ and  $r(v)=d(u)=x,$ and the following composition rules $vu=a,$ and $a^2=u^{-1}u=vv^{-1}=x$ and $v^{-1}v=uu^{-1}=y.$ The following diagram ilustrates the composition in $\G$ 
\[
  \begin{tikzcd}
    x \arrow{r}{u} \arrow[swap]{dr}{a} & y \arrow{d}{v} \\
     & x
  \end{tikzcd}
\]
Set $\mathcal{H}_1=\{u, u^{-1}, x,y\}$ and $\mathcal{H}_2=\{v, v^{-1}, x,y\}.$ It is not diffiult to check that $\mathcal{H}_1$ and $\mathcal{H}_2$ satisfy conditions (i)-(iii) of Proposition  \ref{caracterizacionproductodirecto2}. Then  $\G=\Hh_1 \Hh_2$ is embedded in $\Hh_1\times \Hh_2.$ }
\end{exe}

\subsection{Semidirect product and groupoids}\label{semi}

Let $\mathcal{G}$ be a groupoid, $G$ be a group with identity $1_G$  and $\omega: G\to {\rm Aut}(\mathcal{G})$ be a homomorphism of groups.  The homomorphism $\omega$ induce a action of  $G$ on $\mathcal{G}$ given by
\begin{align*}
\cdot \,: \mathcal{G} \times G &\to \mathcal{G}\\
(x,g) & \mapsto x\cdot g:= \omega_{g^{-1}}(x).
\end{align*} 
We use $\omega$ to define  a groupoid structure on the set $\mathcal{G}\times G.$ Define the partial product as follows:
\begin{equation*}
\exists (x,g)(z,h)\,\,\text{ if and only if}\,\,\,d(x)=r(\omega_g(z)).
\end{equation*}
In this case we set 
\begin{equation*}\label{semip}(x,g)(z,h)=(x\omega_g(z), gh).
\end{equation*}

Here the identities are given  as follows. For each $(x,a)\in \mathcal{G}\times G$,  we  set  $d(x,a)=(d(x)\cdot a, 1_G)$ and $r(x,a)=(r(x),1_G)$. Further, for  $(x,a)\in \mathcal{G}\times G$ we have
\begin{equation}\label{inve}(x,a)^{-1}=(\om_{a^{-1}}(x^{-1}), a^{-1})\in \mathcal{G}\times G.\end {equation}

With this product the groupoid $\mathcal{G}\times G$ is denoted  by $\mathcal{G}\times_{\omega} G$  and is called the \textit{semidirect product of the groupoid $\mathcal{G}$ with the group $G$, via the homomorphism $\omega: G\to Aut(\mathcal{G})$.}

Before giving an example of the construction above we recall that a groupoid $\G$ is connected, if for all $e,f\in \G_0$ there is $g\in \G$ such that $d(g)=e$ and $r(g)=f$. In this case there is an isomorphism $\G\simeq \G_0^2\times \G_e$, where $\G_0^2$ is the coarse groupoid associated to  $\G_0,$ that is   $d(x,y)=(x,x)$ and $r(x,y)=(y,y)$, for all $x,y\in \G_0$. The rule of composition in  $\G_0^2\times \G_e$ is given by: $(y,z)(x,y)=(x,z)$, for all $x,y,z\in \G_0$.
\begin{exe}{\rm Let $\G$ be a connected groupoid and take $e\in \G_0,$ write $\G= \G_0^2\times \G_e.$ Consider an action $\cdot: \G_e\times \G_e\to \G_e$ then $\G_e$ acts on $\G$ via the homomorphism $\om: \G_e\to {\rm Aut}(\G),$ where for $a\in \G_e,$ $ \om_a(x,y, g)=(x,y, a\cdot g),$ for all $((x,y),g)\in \G.$ Now $(\G_0^2\times \G_e)_0=\{(x,x, e)\mid x\in \G_0\}$ and $r(\omega_a(x,y, g))=(y,y,e).$  Then $\G\times_\om \G_e$ is a groupoid with partial  product 
$$(y,z,g, a)(x,y,h,b)=((y,z,g)(x,y,a\cdot h),ab)=(x,z, g(a\cdot h), ab).$$
for all $(y,z,g), (x,y,h)\in \G$ and $a,b\in \G_e.$}
\end{exe}

The following result characterizes when a semidirect product is a direct product.
\begin{thm} \label{semipro}Let $\mathcal{G}$ be a groupoid and $G$ be a group with identity $1_G$  and $\omega: G\to Aut(\mathcal{G})$ a homomorphism of groups. Then the following assertions are equivalent.
\begin{enumerate}
\item The identity map between $\G\times G$ and $\mathcal{G}\times_{\omega} G$ is a homomorphism, and thus an isomorphism.
\item $\omega$ is trivial.
\item $\G_0 \times G$ is normal in $\mathcal{G}\times_{\omega} G.$ 
\end{enumerate}
\end{thm}
\begin{proof}
(i) $\Rightarrow $ (ii) Take $g\in G$ we need to show that $\omega_g$ is the identity map on $\G.$ Take $z\in \G, $ then $\exists (r(z), g)(z,b) \in \G\times G ,$ for $b\in G.$ Since the identity map is a homomorphism of groupoids we have that $\exists (r(z), g)(z,b)$ and $(z,gb)=(r(z)\om_a(z), gb)$. In particular $z=\om_g(z)$ and $\om$ is trivial.

(ii) $\Rightarrow $ (i) Si $\om$ is trivial, then $\om_g$ is the identity map on $\G$ for all $g\in G$. Hence $\mathcal{G}\times_{\omega} G=\G\times G$ and  the identity map is an isomorphism.

(ii) $\Rightarrow $ (iii) It is not difficult to check that $\G_0 \times G$ is a subgroupoid of $\mathcal{G}\times_{\omega} G,$ also   $(\G_0 \times G)_0=\G_0\times\{1_G\}=(\mathcal{G}\times_{\omega} G)_0,$ then $\G_0 \times G$ is  wide. 
Let $(x,g)\in \mathcal{G}\times_{\omega} G $ and $(e, h)\in \G_0 \times G$ such that $\exists (x,g)(e, h)(x,g)^{-1}\in \mathcal{G}\times_{\omega} G,$ but $\omega$ is trivial, then 
\begin{align*} (x,g)(e, h)(x,g)^{-1}&\stackrel{\eqref{inve}}=(xe, gh)(x^{-1},g^{-1})=(xex^{-1}, ghg^{-1})=(e,ghg^{-1})\in \mathcal{G}_0\times G.
\end{align*}
(3) $\Rightarrow $ (2) Suppose that $\G_0 \times G$ is normal in $\mathcal{G}\times_{\omega} G$. Take $l\in G,$  $x\in \G$ and write $l= ghg^{-1}$ for some $g,h\in G.$  Now  $(\om_{g^{-1}}(d(x)), h)\in \G_0 \times G$  and   $\exists (x,g)(\om_{g^{-1}}(d(x)), h)\in \mathcal{G}\times_{\omega} G$ which equals $(x,gh)$.  Moreover by \eqref{inve} $(x,g)^{-1}=(\om_{g^{-1}}(x^{-1}), g^{-1})$ and 
$$r(\om_g(\om_{g^{-1}}(x^{-1})))=r(x^{-1})=d(x),$$ thus  $\exists (x,g)(\om_{g^{-1}}(d(x)), h)(x,g)^{-1}$ and 
\begin{align*}(x,g)(\om_{g^{-1}}(d(x)), h)(x,g)^{-1}&\stackrel{\eqref{inve}}=(x,gh)(\omega_{g^{-1}}(x^{-1}),g^{-1})\\&=(x\om_{ghg^{-1}}(x^{-1}), ghg^{-1})\in \G_0 \times G,
\end{align*} from this $x\om_{l}(x^{-1})=x\om_{ghg^{-1}}(x^{-1})\in \G_0$ then  $x=\om_l(x)$ and $\om$ is trivial.
\end{proof}
The following result tells us how to recognice  semidirect products inside a groupoid.

\begin{prop} Let $\G$ be a groupoid, $G$ be a subgroup of $\G$ and $\Hh$ a normal subgroupoid of $\G.$ Suppose that $\Hh\cap G=\{1_G\}$ and $G$ acts in $\Hh$ via  conjugation, then $\Hh G\simeq \Hh\times_w G.$ 
\end{prop}
\begin{proof} First of all by \cite[Proposition 11]{AMP}, $\Hh G$  is a subgroupoid of $\G.$ Since $\Hh\cap G=\{1_G\}$ every element in   $\Hh G$ has a unique expresion $hg,$ for some $h\in \Hh, g\in G$ with $d(h)=r(g),$ and thus the map $\varphi: \Hh G\ni hg\mapsto (h,g)\in  \Hh\times_w G$ is a bijection. To check that it is an isomorphism take $h,h'\in \Hh, g,g'\in G$ such that $\exists hg, \exists h'g'$ and $\exists (hg)(h'g'),$ then $\exists gh', \exists gh'g^{-1}, \exists hgh'g^{-1}$ and 
\begin{align*}\varphi((hg)(h'g'))&=\varphi((hgh'g^{-1})(gg'))\\
&=\varphi((h\om_g(h'))(gg'))\\
&=(h\om_g(h'),gg')\\
&=\varphi(hg)\varphi(h'g').
\end{align*}
Hence $\Hh G\simeq \Hh\times_w G.$ 
\end{proof}

\section{Partial actions and groupoids}\label{paction}
In this section we construct a groupoid  associated to a partial action of $\G$  on a set $X.$ We recall  from \cite{NY} the following.

 \begin{defn}
Let $X$ be a set and $\G$ be a groupoid. A \textit{partial groupoid action} of $\G$ on $X$ is a partially  defined function $\alpha:D\subseteq \G\times X\to X$ given by $(g,x)\mapsto g\cdot x$, for all $(g,x)\in D,$  when $g\cdot x$ is defined we write $\exists g\cdot x$ and the following axioms hold:

\begin{enumerate} [\quad\rm (PGrA1)] 
	\item[(PGrA1)] For each $x\in X$, there is $e\in \G_0$ such that $\exists e\cdot x$. If $f\in \G_0$    and $x\in X$, $\exists f\cdot x$ implies $f\cdot x=x$.
	\item[(PGrA2)] $\exists g\cdot x$, implies $\exists g^{-1}\cdot (g\cdot x)$, and $g^{-1}\cdot (g\cdot x)=x$.
	\item[(PGrA3)] $\exists gh$ and $\exists g\cdot (h\cdot x)$, implies $\exists (gh)\cdot x$ and $g\cdot (h\cdot x)=(gh)\cdot x$.
\end{enumerate}
\hspace{0,3cm}We say that  $\alpha$ is global if also satisfies.
\begin{enumerate}[\quad\rm (PGrA4)] 
      \item[(PGrA4)] given $g \in \G$ and $x \in X$ such that $\exists d(g)\cdot x$ then $\exists g\cdot x.$
\end{enumerate}
\end{defn}

The following result will be useful in the sequel
\begin{lem} \label{exxi}Suppose that $\G$ acts partially on $X,$ then  \begin{align}\label{equiva}
\exists g\cdot x 
&\Longleftrightarrow(r(g),g\cdot x)\in  ( \mathcal{G}, X),
\end{align} $g\in \G$ and $x\in X.$
\end{lem}

\begin{proof} Take $g\in \G$ and $x\in X.$ Then
\begin{align*}
\exists g\cdot x &\Longleftrightarrow \exists (gg^{-1}g)\cdot x\Longleftrightarrow \exists g\cdot [g^{-1}\cdot (g\cdot x)]\Longleftrightarrow \exists (gg^{-1})\cdot (g\cdot x)
\Longleftrightarrow \exists r(g)\cdot (g\cdot x)
\end{align*} 
\end{proof}

We recall from \cite[p. 188]{G} the definition of the {\it action groupoid} induced from a  groupoid $\G$  acting partially on a set $X.$ Consider the set
\begin{equation}\label{actiongroup}{(\mathcal{G}, X)}=\{(g,x)\in \mathcal{G}\times X \mid \exists g\cdot  x \},\end{equation}
with partial product defined  as follows:
take  $(g,x),(h,y)\in (\G, X)$ then   $\exists (g,x) (h,y)$ if and only if 
$d(g)=r(h)$ and $x=h\cdot y$. In this case   
\begin{equation}\label{parprod}(g,x)(h,y)=(gh,y). \end{equation}
From this we get $(\mathcal{G}, X)_0=\{(e,x)\in \G_0\times X\mid \exists\, e\cdot x\}$,  and for $(g,x)\in(\G, X),$ 
\begin{equation}\label{domain} d(g,x)=(d(g),x)
\end{equation} and \begin{equation} \label{rango}r(g,x)=(r(g),g\cdot x).
\end{equation}

Given a partial action  of  a group $G$ on a set $X$  in \cite[p. 1041]{A2} it was introduced a groupoid associates to $G$ and $X,$ their construction can be extended to the context of partial grupoid actions as follows. Let $\alpha$ be a partial action of $\G$ on $X,$ the graph of $\alpha$ is the set   $\gra=\{(g, x, y)\mid \exists g\cdot x \wedge g\cdot x=y\}.$
%
Then $\gra$ has a groupoid structure where $\gra_0=\{(e,x,x)\mid \exists e\cdot x\}$ and $(g, x, y):(r(g), y, y)\to (d(g), x, x),$ the composition rule in $\gra$ is given by 
$\exists (h,v,w)  (g, x, y)$ if $\exists gh$ in $\G$ and $x=w,$ in this case 
$$(h,v,x)  (g, x, y)=(gh, v, y),$$ since $g\cdot x=y, h\cdot v=x$ and $\exists gh$ we have $r(gh)=r(g),$ $d(gh)=d(h)$ and  by  (PGrA3) that $(gh)\cdot v=y$ and so  the product is well defined;   the associativity of the composition
law follows from the associativity of the multiplication in $\G$.  Moreover $(g, x,y)^{-1}=(g^{-1}, y, x).$ The functor $\F: (\G, X)\ni (g,x) \mapsto (g, x, g\cdot x)\in \gra $  determines a groupoid isomorphism, then we have the next.

\begin{prop} Let $\alpha$ be a partial action of $\G$ on $X,$ then the groupoids $(\G, X)$ and $\gra$ are isomorphic.
\end{prop}

Recall that a fuctor $\Gamma:\G\to \G'$ is {\it star injective} if, $\Gamma(x)=\Gamma(y)$ and $d(x)=d(y)$ implies $x=y.$ It is {\it star surjective} if for $e\in \G_0$ and $g'\in \G'$ such that $\Gamma(e)=d(g')$ there is $g\in \G$ such that $d(g)=e$ and $\Gamma(g)=h.$ We say that $\Ga$ is a covering if it is star injective and star surjective.

 \begin{prop}\cite[Prop. 4.7]{G}\label{partostar} Let $\G$ and $X$ and  as above, then with the partial product given by  \eqref{parprod}, the set ${(\mathcal{G}, X)}$ is a groupoid. Moreover the map 
\begin{equation}\label{gam}\Gamma: {(\mathcal{G}, X)}\ni (g,x)\mapsto g\in \G\end{equation} is star injective, and $\Gamma$ is a covering if and only if $\G$ acts globally in $X.$
\end{prop}

We provide a converse of Proposition \ref{partostar}.

\begin{prop}\label{startopar}Let $\mathcal{H}$ be a groupoid and $\Gamma: \mathcal{H}\to \mathcal{G}$ be a star injective functor, define a partial function from $\mathcal{G} \times \Hh_0$ to $\Hh_0$ by setting:
\begin{equation}\label{condition}\exists g\cdot x \Longleftrightarrow (\exists h\in\mathcal{H})(d(h)=x\,\wedge \Gamma(h)=g ),
\end{equation}
in this case we set 
\begin{equation}\label{result}g\cdot x =r(h).\end{equation} This defines a partial action of $\G$ on $\Hh_0$, and the action is global if and only if $\Gamma$ a covering.
\end{prop}
\begin{proof} For (PGrA1) take $x\in\Hh_0,$ then  then $d(x)=x$ and $\Gamma(x)\in \G_0,$ thus $\exists \Gamma(x)\cdot x.$ Now if $f\in \G_0$ and $\exists f\cdot x,$ then there is $h\in\Hh_0$ such that $d(h)=x$ and $\Gamma(h)=f$ since $f\in \G_0$ we have $\Gamma(d(h))=f=\Gamma(h)$ and the fact that $\Gamma$ is star injective implies $d(h)=h,$ thus $f\cdot h=r(h)=d(h)=x.$ Conditions (PGrA2) and (PGrA3) are proved in a similar way as in \cite[Proposition 3.7]{KL}.

Suppose that $\G$ acts globally on $\Hh_0,$ let $e\in \Hh_0$  and $g\in \G$ such that $\Gamma(e)=d(g),$ then by  \eqref{condition} $\exists d(g)\cdot e,$ since the action is global we have $\exists g\cdot e$ and thus there is $h\in \Hh$ such that $\Gamma(h)=g,$ which shows that $\Gamma$ is star surjective. Conversely, suppose that $\Gamma$ is star surjective and take $g\in G$ and $e\in\Hh_0$ such that $\exists d(g)\cdot x,$  by \eqref{condition} there is $h\in \Hh$ such that $d(h)=x$ and $\Gamma(h)=d(g),$ then $\Gamma(d(h))=d(g)$ and the  fact that $\Gamma$ is star surjective implies that there is $l\in \Hh$ such that $d(l)=d(h)=x$ and $\Gamma(l)=g,$ again \eqref{condition} implies that $\exists g\cdot x$ and the action is global.
\end{proof}
\subsection{On the category of partial actions of $\G$}
We denote by  $\pacg$ the category whose objects are partial actions of $\G$ on sets. Given two partial actions $\alpha: \G\times X\to X$ and  $\alpha':\G\times X'\to X'$ on the sets $X$ and $X',$ respectively, a morphism $f\colon\alpha\to \alpha' $  is a map $f: X\to X'$ such that  for $g\in G$ and $x\in X$ such that $\exists g\cdot x,$ then $\exists g\cdot f(x)$ and  $f(g\cdot x)=g\cdot f(x).$  On the other hand the category $\stig$ consists of all star injections to $\G,$ and given $\Ga_1: \G_1\to\G, \Ga_2:\G_2\to \G$ two star injections, a morphism $f:\Ga_1\to \Ga_2$ is a functor  $f:\G_1\to \G_2$ such that $\Ga_2\circ f=\Ga_1.$

Let $\alpha$ be a partial action of $\G$ on a set $X,$ for $g\in \G$ set $X_g=\{x\in X\mid \exists g^{-1}\cdot x\},$  by (PGrA1) we have $X=\bigcup_{e\in \G_0}X_e.$ We say that $\alpha$ is {\it strict} if $X_e\cap X_f=\emptyset$ for $e\neq f$ and in this case there is a disjoint union $X=\bigsqcup\limits_{e\in \G_0}X_e.$  Denote by $\spacg$ the full subcategory of $\pacg$ whose objects are strict partial actions.

We proceed with the next.
\begin{thm}  \label{equiv}
The categories $\spacg$ and $\stig$ are equivalent.
\end{thm}
\begin{proof}  Let  $\F\colon \spacg\to \stig$ be a functor defined as follows. In the level of objects we set  $\F(\alpha)=\Gamma^\alpha,$ where $\Gamma^\alpha$ is  the star injective functor given by \eqref{gam},  for a morphism $f:\alpha\to \alpha'$ we set  $\F(f):\Gamma^\alpha\to \Gamma^{\alpha'},$ by $$F(f): (\G, X)\ni (g,x)\mapsto (g, f(x))\in (\G, X').$$
 Conversely, let $\Gg\colon \stig\to \pacg$ such that  for any star injective functor $\Gamma:\Hh\to \G,$ we set  $\Gg(\Gamma)=\alpha^\Gamma,$ where $\alpha^\Gamma$ is the partial action given by  \eqref{condition} and \eqref{result}. To check that $\alpha^\Gamma$ is strict take $e,f \in \G_0$ and $x\in X$ such that $\exists e\cdot x$ and $\exists f\cdot x,$ then by  \eqref{condition} there are $h_1, h_2\in \Hh$ such that $d(h_1)=d(h_2)=x, \Ga(h_1)=e$ and $\Ga(h_2)=f.$ Thus  $e= \Ga(d(h_1))=\Ga(d(h_2))=f,$ and  $\alpha^\Gamma$ is strict.

Now let $\Ga_1: \Hh_1\to \G$ and $\Ga_2: \Hh_2\to \G$ be objects in $\stig$ and $f: \Ga_1\to \Ga_2$ a morphism. We set $\Gg(f):\alpha^{\Gamma_1}\to \alpha^{\Gamma_2} ,$ where 
$$\Gg(f):(\Hh_1)_0 \ni e\mapsto f(e)\in (\Hh_2)_0.$$ To check that $\Gg(f)$ is a morphism in  $\pacg,$ take $g\in \G$ and $e\in (\Hh_1)_0$ such that $\exists g\cdot e,$  then there is $h\in \Hh$ with $d(h)=e$ and $\Ga_1(h)=g.$ Now $d(f(h))=f(e)$ and $\Ga_2(f(h))=g$ thus $\exists g\cdot f(e),$ and by \eqref{result} $$f(g\cdot e)=f(r(h))=r(f(h))=g\cdot f(e).$$  We shall check that the pair $(\F, \Gg)$ determines an equivalence between the categories $\spacg$ and $\stig$. Let $\alpha$ be a partial action of $\G$ on $X,$ notice that $ \alpha^{\Gamma^{\alpha}}$ is a strict partial action on $(\G, X)_0$ 
and define  $\tau_\alpha: \alpha\to \alpha^{\Gamma^{\alpha}}$ by $$\tau_\alpha: X\ni x\mapsto (e_{x}, x)\in (\G_0, X), $$ where $e_x$ is the only element in $\G_0$ such that $\exists e_x\cdot x.$ We show that $\tau_\alpha$ is a morphism. Take $g\in \G$ and $x\in X$ such that $\exists g\cdot x,$ by (PGrA2) and (PGrA3) $\exists d(g)\cdot x$  and the fact that $\alpha$ is strict implies $d(g)=e_x,$ thus   $\exists g\cdot (e_x, x)$ because  $d(g,x)\stackrel{ \eqref{domain}}=(e_x, x)$ and $\alpha^\Gamma(g,x)=g.$ Moreover 
$$\tau_\alpha(g \cdot x )=(e_{g \cdot x}, g \cdot x)\stackrel{\eqref{equiva}}=(r(g), g \cdot x)\stackrel{{\eqref{rango}}}=r(g,x)=g\cdot (e_\alpha, x)\stackrel{\eqref{result}}=r(g,x) =g\cdot \tau_\alpha(x).$$ Since the maps $\tau_\alpha, \alpha$ in $ \spacg$ are bijections  and 
 the  diagram 
\[
\begin{diagram}
\node{\alpha}\arrow{e,t}{\tau_\alpha}\arrow{s,l}{f} \node{\Gg\F(\alpha)} \arrow{s,r}{\Gg\F(f)}\\
\node{\alpha'} \arrow{e,b}{\tau_{\alpha'}} \node{\Gg\F(\alpha')}
\end{diagram}
\]
commutes, we conclude that $\tau: {\bf id}_{\spacg}\to\Gg\F$ is a natural isomorphism.  It remains to construct a natural isomorphism  $\eta:  {\bf id}_{\stig}\to\F\Gg.$  Let $\Gamma : \Hh \to \G$ be a star injective functor, by  Proposition \ref{partostar}  the map $\Gamma^{\alpha^{\Gamma}}: (\G, \Hh_0)\to \G$ is star injective. Define $\eta_\Gamma: \Gamma \to \Gamma^{\alpha^{\Gamma}}$ by
$$\eta_\Gamma: \Hh \ni h \mapsto (\Gamma(h),d(h))\in (\G, \Hh_0).$$
We check that $\eta_\G$ is a morphism. First to see that  $\eta_\Gamma:\Hh\to (\G,\Hh_0)$ is a groupoid homomorphism, take $h_1, h_2\in \Hh$ such that  $\exists h_1h_2$ then
\begin{align*}
\eta_\Gamma(h_1h_2)&=(\Gamma(h_1h_2),d(h_1h_2))=(\Gamma(h_1)\Gamma(h_2),d(h_2))\\
&\stackrel{\eqref{parprod}}=(\Gamma(h_1),d(h_1))(\Gamma(h_2),d(h_2))=\eta(\Gamma)(h_1)\eta(\Gamma)(h_2). 
\end{align*}
Further, for $h\in \Hh,$ $(\Gamma^{\alpha^{\Gamma}}\circ \eta_\Gamma)(h)=\Gamma^{\alpha^{\Gamma}}(\Gamma (h),d(h))=\Gamma (h).$ Now we check that $\eta_\Ga$ is a bijection, indeed   $\eta_\Gamma$ is injective because $\Gamma$ is star injective, and if $(g,x)\in (G, \Hh_0)$ then by  \eqref{actiongroup} $\exists g\cdot x$  which thanks to \eqref{condition} means that there is  $h\in\mathcal{H}$ such that $d(h)=x$ and $ \Gamma(h)=g,$ then $(g,x)=\eta_\Ga(h) $ and $\eta_\Ga$ is surjective.  Moreover it is not difficult to see that the diagram
\[
\begin{diagram}
\node{\Ga}\arrow{e,t}{\eta_\Ga}\arrow{s,l}{f} \node{\F\Gg(\Ga)} \arrow{s,r}{\F\Gg(f)}\\
\node{\Ga'} \arrow{e,b}{\eta_{\Ga'}} \node{\F\Gg(\Ga')}
\end{diagram}
\]commutes, which ends the proof.
 \end{proof}
It is clear that when $\G$ is a group, then  the categories $\spacg$ and $\pacg$ coincide, thus $\pacg$  and $\stig$ are equivalent as observed in \cite[p.  102]{KL}. Moreover denoting by $\sgacg$ and $\cov$ the subcategories of $\spacg$ and  $\stig$ consisting of strict global actions and coverings  to $\G,$ respectively,  by  Proposition \ref{partostar}, \ref{startopar} and Theorem \ref{equiv} we have.
 \begin{cor}  The categories $\sgacg$ and $\cov$ are equivalent.
\end{cor}
 
Let $\C$ be a category, we say that a functor $\F: \C\to {\bf Set}$ is {\it strongly injective} on objects if $\F(e)\cap \F(f)=\emptyset,$ for all $e,f\in \C_0,$ with $e\neq f.$ Then by \cite[Proposition 10]{NY} and its proof  we have the next.

\begin{prop} Strict global actions of $\G$ correspond to strongly injective functors $\F:\G\to {\bf Set}.$
\end{prop}



\subsection{A relation with globalization}

Recall from \cite{NY} the definition  of globalization. 
 Let $\alpha$ be a partial action of $\G$ on $X$. A {\it globalization} of $\alpha$ is a pair $(\beta, Y)$ where $\beta:\G\times Y\to Y$ is a global action and there is a morphism $\iota: \alpha\to \beta.$ We say that $(\beta, Y)$ is a {\it universal globalization} if for any other globalization $(\beta', Z)$ with morphism  $j: \alpha\to \beta'$ there is a unique   morphism $k: \beta\to \beta' $ such that $j=k\circ \iota.$  If follows from \cite[Theorem 4]{NY} that every partial category action on a set admits a universal globalization, this implies that the category $\gacg$ of global actions of $\G$ is reflective in $\pacg$ (see \cite[IV. 3]{MA}). In particular two  universal globalizations are unique up to isomorphism.

Suppose that $\G$ act partially on the set $X$ and let $\Gamma : (\G,X)\to \G$  be the  star injective functor defined in  Proposition \ref{partostar}. Let $Y$ be the universal  globalization of $\alpha$ as given in  \cite[Definition 18]{NY} and  $i:X\ni x \mapsto [e,x] \in Y,$ where $e\in \G_0$ is  such that $\exists e\cdot x$.  By \cite[Theorem 4]{NY}  the map $i$ is well defined  and is a morphism $i:\alpha\to \beta,$ where $\beta$ is the global action of $\G$ on $Y$ constructed in \cite[Definition 20]{NY}.
Let $\Pi: (\G,Y)\to \G$ be the corresponding covering functor. It is posible to define a functor $\nu: (\G,X)\to (\G,Y)$ given by $(g,x)\mapsto \nu(g,x)=(g, i(x)),$ it is not difficult to see that $\nu$ is injective and  $\Pi\nu=\Gamma.$

Recall that a subcategory $\D$ of a category $\C$ is dense if  each identity in $\C$ is isomorphic to an identity in  $\C.$

We finish this work with the following.

\begin{prop}
With the notation above $\nu((\G,X))$ is full and dense in $(\G,Y)$.
\end{prop}

\begin{proof}
First of all we prove that $\nu((\G,X))$ is full in $(\G,Y)$. In fact, suppose that $\exists gh$ and take $(g,[h,x])\in (\G,Y)$  such that $d(g,[h,x]), r(g, [h,x])\in \nu((\G,X)),$ we need to show that $(g,[h,x])\in \nu((\G,X)).$ Indeed, as $[h,x], [gh, x]\in i(X)$ then $\exists h\cdot x$ and $\exists (gh)\cdot x$. On the other hand  Lemma \ref{exxi}  tell us that $[l,x]=[r(l),l\cdot x],$ provide that $\exists l\cdot x,$ then 
$(g,[h,x])=(g, [r(h),h\cdot x])$ and $g\cdot [h,x]=[gh, x]=[r(gh), (gh)\cdot x],$ from this $(g,[h,x])  \in \nu((\G,X))$. To prove that $\nu(\G,X)$ is dense in $(\G,Y)$ take  an identity  $(r(g),[g,x])\in (\G,Y)_0$. Then $(g,[d(g),x])\in (\G,Y)$ and  $r(g,[d(g),x])\stackrel{\eqref{rango}}=(r(g),[g,x])$ and $d((g,[d(g),x]))=(d(g),[d(g),x])\in \nu((\G,X)),$ thus$(r(g),[g,x])$ is isomorphic to $(d(g),[d(g),x]).$
\end{proof}


\begin{thebibliography}{54}
\bibitem{A2} F. Abadie Partial actions and groupoids, \emph{Proc. Am. Math. Soc}, \textbf{132} (2004) 1037-1047.

\bibitem{A1} F. Abadie Partial actions of discrete groups and related structures, \emph{Publicaciones Matem\'aticas del Uruguay}, \textbf{10} (2005) 1-9.



\bibitem{AMP} J. \'{A}vila, V. Mar\'{i}n and H. Pinedo, Isomorphism Theorems for Groupoids and Some Applications, {\it International Journal of Mathematics and Mathematical Sciences} \textbf{2020} (2020) 1--10.




\bibitem{BP} D. Bagio and A. Paques, Partial Groupoid Actions: Globalization, Morita Theory, and Galois Theory, \emph{Comm. Algebra} \textbf{40} (2012) 3658-3678.



\bibitem{Br} R. Brown, From groups to groupoids: a brief survey, \emph{Bull. Lond. Math. Soc.} \textbf{19} (1987) 113--134.


\bibitem{L} M. V. Lawson (1998), Inverse Semigroups: The Theory of Partial Symmetries, \emph{World Scientific}, Singapore.


\bibitem{G} N. D. Gilbert, Actions and expansions of ordered groupoids, \emph{J. Pure Appl. Algebra} \textbf{198} (2005) 175-195.

\bibitem{MP}  V. Mar\'{i}n and H. Pinedo, Partial  groupoid actions on $R$-categories: globalization and the smash product, {\it J. Alg. Appl.} \textbf{19} (5) (2020) 2050083.


\bibitem{KL} J. Kellendonk and M. V. Lawson, Partial Actions of Groups, \emph{International Journal of Algebra and Computation} \textbf{14} (2004) 87-114.

\bibitem{MA} S. Mac Lane, Categories for theWorking Mathematician (Springer, New York–Heidelberg–Berlin,
1971).

\bibitem{NY}
P. Nystedt,
Partial category actions on sets and topological spaces,
{\it Comm. Algebra} \textbf{46}(2) (2018) 671--683.


\bibitem{NOP}
P. Nystedt, J. \"{O}inert  and H. Pinedo, Epsilon-strongly groupoid graded rings, the Picard inverse category and cohomology, {\it Glasgow Math. J.} \textbf{62} (2020) 233--259..


\bibitem{PL} K. E. Pledger, Internal Direct Product of Groupoids, {\it J. Algebra}, \textbf{509} (1999) 599--627.


\bibitem{PT} A. Paques, and  T. Tamusiunas, The Galois correspondence theorem  for  groupoid actions, {\it J. Algebra} \textbf{217},(2018) 105--123.





\end{thebibliography}
\end{document}